\documentclass[11pt]{amsart}
\usepackage{amsmath}
\usepackage{amsthm}
\usepackage{amssymb}
\usepackage{a4wide}
\usepackage{color}

\usepackage[pagewise,mathlines]{lineno}
\usepackage{hyperref}
\usepackage{amsfonts}

\newtheorem{lemma}{Lemma}[section]

\newtheorem{theorem}{Theorem}[section]

\newtheorem{remark}{Remark}[section]

\def\rr{\mathbb{R}}
\def\nn{\mathbb{N}}

\def\rr{\mathbb{R}}

\def\eps{\varepsilon}

\def\rrd{\rr^d}

\begin{document}

\title[Estimates for nonlinear nonlocal diffusion problems]
{Decay estimates for nonlinear nonlocal diffusion problems in the
whole space}

\author[L. I. Ignat, D. Pinasco,
J. D. Rossi and A. San Antolin]{Liviu I. Ignat, Dami\'an Pinasco,
Julio D. Rossi and Angel San Antolin}

\address{L. I. Ignat
\hfill\break\indent Institute of Mathematics ``Simion Stoilow'' of
the Romanian Academy, \hfill\break\indent 21 Calea Grivitei Street
 \hfill\break\indent 010702, Bucharest,  ROMANIA \hfill\break\indent
and \hfill\break\indent
 BCAM - Basque Center for Applied Mathematics, \hfill\break\indent
 Mazarredo 14, 48009, Bilbao,  SPAIN.}
 \email{{\tt
liviu.ignat@gmail.com}\hfill\break\indent  {\it Web page: }{\tt
http://www.imar.ro/\~\,lignat}}

\address{D. Pinasco
\hfill\break\indent Departamento de Matem\'aticas y Estad\'{\i}stica,
Universidad Torcuato di Tella  \hfill\break\indent
Mi\~nones 2177, C1428ATG, Ciudad Aut\'onoma de Buenos Aires, ARGENTINA \hfill\break\indent
and \hfill\break\indent Consejo Nacional de Investigaciones Cient\'{\i}ficas y T\'ecnicas (CONICET), \hfill\break\indent Ciudad Aut\'onoma de Buenos Aires, ARGENTINA}
\email{{\tt dpinasco@utdt.edu}}

\address{J. D. Rossi
\hfill\break\indent Departamento de An\'{a}lisis Matem\'{a}tico,
Universidad de Alicante, \hfill\break\indent Ap. correos 99, 03080,
\hfill\break\indent Alicante, SPAIN. \hfill\break\indent On leave
from \hfill\break\indent Dpto. de Matem{\'a}ticas, FCEyN,
Universidad de Buenos Aires, \hfill\break\indent 1428, Buenos Aires,
ARGENTINA. } \email{{\tt jrossi@dm.uba.ar} \hfill\break\indent {\it
Web page: }{\tt http://mate.dm.uba.ar/$\sim$jrossi/}}

\address{A. San Antolin
\hfill\break\indent Departamento de An\'{a}lisis Matem\'{a}tico,
 Universidad de Alicante,
\hfill\break\indent Ap. correos 99, \hfill\break\indent 03080,
Alicante, SPAIN. } \email{{\tt angel.sanantolin@ua.es}}


\begin{abstract} In this paper we obtain  bounds for the decay rate in the $L^r (\rr^d)$-norm for the solutions to a nonlocal and nolinear
evolution equation, namely, $$u_t(x,t) =  \int_{\rr^d} K(x,y) |u(y,t)- u(x,t)|^{p-2} (u(y,t)- u(x,t)) \, dy, $$ with $ x \in \rr^d$, $ t>0$.
Here we consider a kernel $K(x,y)$ of the form $K(x,y)=\psi (y-a(x))+\psi(x-a( y))$, where $\psi$ is a bounded, nonnegative function supported in the unit ball  and $a$ is a linear function $a(x)= Ax$.
To obtain the decay rates we derive
 lower and upper bounds for the first eigenvalue of a nonlocal diffusion operator of the form $ T(u) = - \int_{\rr^d} K(x,y) |u(y)-u(x)|^{p-2} (u(y)-u(x)) \, dy$, with $1 \leq p < \infty$.  The upper and lower bounds that we
obtain are sharp and provide an explicit expression for the first eigenvalue in the whole $\rr^d$:
$$
\lambda_{1,p} (\rr^d) = 2 \left( \int_{\rr^d} \psi (z) \, dz \right)
\left| \frac{1}{|\det{A}|^{1/p}} -1 \right|^p.
$$
Moreover, we deal with the $p=\infty$ eigenvalue problem studying  the limit as $p \to \infty$ of $\lambda_{1,p}^{1/p}$.
\end{abstract}

\maketitle

\section{Introduction}
\label{Sect.intro}
\setcounter{equation}{0}

Recently, nonlocal problems have been  widely used to model
diffusion processes. In particular, for $J: \rr^N \to \rr$ a
nonnegative, radial, continuous function with $\int_{\rr^N} J(z)\,
dz =1$,  nonlocal evolution equations of the form
\begin{equation} \label{11}
u_t (x,t) = (J*u-u) (x,t) = \int_{\rr^N} J(x-y)u(y,t) \, dy-u(x,t),
\end{equation}
and variations of it, have been widely used to model
diffusion processes. As is stated in \cite{F}, if $u(x,t)$ is thought
of as a density at the point $x$ at time $t$, and $J(x-y)$ is
thought of as the probability distribution of jumping from
location $y$ to location $x$, then $\int_{\rr^N} J(y-x)u(y,t)\, dy
= (J*u)(x,t)$ is the rate at which individuals are arriving at
position $x$ from all other places and $-u(x,t) = -\int_{\rr^N}
J(y-x)u(x,t)\, dy$ is the rate at which they are leaving location
$x$ to travel to all other sites. This consideration, in the
absence of external or internal sources, leads immediately to  the
fact that the density $u$ satisfies equation (\ref{11}).

Equation (\ref{11}) is called {\it nonlocal diffusion equation} since
the diffusion of the density $u$ at a point $x$ and time $t$ does
not only depend on $u(x,t)$ and its derivatives, but on all the
values of $u$ in a neighborhood of $x$ through the convolution
term $J*u$. This equation shares many properties with the
classical heat equation, $u_t= \Delta u$, such as: bounded
stationary solutions are constant, a maximum principle holds for
both of them and, even if $J$ is compactly supported,
perturbations propagate with infinite speed (see \cite{F} for more details). However,
there is no regularizing effect in general.

Here we deal with a nonlinear nonlocal problem, analogous
to the classical $p-$Laplacian evolution equation, $u_t=
\mbox{div} (|\nabla u|^{p-2} \nabla u)=\Delta_p u $, namely,
\begin{equation}
\label{eq.parabolica}
\left\{
\begin{array}{ll}
u_t(x,t) =  \displaystyle\int_{\rr^d} K(x,y) |u(y,t)- u(x,t)|^{p-2} (u(y,t)- u(x,t)) \, dy,& x \in \rr^d, \ t>0, \\[10pt]
u(x,0)=u_0 (x),& x\in \rr^d,
\end{array}
\right.
\end{equation}
with an initial condition $u(x,0)=u_0 (x)$.

The references \cite{AMRT2}, \cite{AMTR3} and \cite{AMTR4} are especially related to the nonlocal problem \eqref{eq.parabolica}. In fact, this work can be viewed as a natural continuation of those papers. All these results were collected in the recent book \cite{libro}. Also, the papers \cite{Sh:10} and \cite{Co:10} deal with the  eigenvalue problem for a general linear nonlocal equation.

Note that here we have a kernel $K(x,y)$ which is not of
convolution type. We will assume a special form for this kernel,
see \eqref{forma.nucleo} below. Existence and uniqueness of
solutions for $u(x,0)=u_0 (x) \in L^1 (\rr^d) \cap  L^\infty
(\rr^d) $ can be obtained as in
\cite[Chapter 6]{libro}, hence our main aim here is to deal with the asymptotic behaviour
as $t\to \infty$. As it is well known, to study the decay of
solutions as $t\to
\infty$ the first eigenvalue of the associated elliptic
part plays a crucial role. Hence, one of our main purposes here is to
study properties of the principal eigenvalue of nonlocal diffusion
operators when the associated kernel is not of convolution type.
We recall that the particular case $p=2$ has been previously treated in \cite{IR3}.

Let us now sate the main assumptions we will use along this paper. We assume some structure for the kernel. Let us consider a nonnegative and bounded function $\psi$, supported in the unit ball of $\rr^d$. In this work we fix the support the unit ball but any compact set can be handled in the same way.  We associate with this function a kernel of the
form
\begin{equation}\label{forma.nucleo}
K(x,y)=\psi (y-a( x))+\psi(x-a( y)), \quad a(x)=Ax,
\end{equation}
where $A$ is an invertible matrix. Note that $K$ is
symmetric and any convolution type kernel also take the form
\eqref{forma.nucleo} (just take $a(x)=x$). For this kernel let us
look for the first eigenvalue of the associated nonlocal operator,
that is,
$$
\lambda_{1,p} (\rr^d)=\inf _{u\in L^p(\rr^d)}\frac{\displaystyle
\int_{\rr^d}\int_{\rrd} K(x,y)|u(x)-u(y)|^pdxdy}{\displaystyle\int _{\rrd}|u|^p(x)dx}.
$$
Due to the lack of compactness it is not known if the infimum is achieved. Hence we do not have an existence result for eigenfunctions, but we still call  $\lambda_{1,p}(\rr^d)$ {\it the first eigenvalue} for this problem because it is defined in an analogous way the local case.

The first  result of this paper is the following:
\begin{theorem} \label{teo.cota.por.abajo}
Let $A \in \rr^{d\times d}$ be an invertible matrix and assume that the kernel $K(x,y)$ is given by \eqref{forma.nucleo}. Then, for $1 \leq p < \infty$, we have
$$
\lambda_{1,p} (\rr^d) = 2 \left| \frac{1}{|\det A|^{1/p}}- 1 \right|^p
\left(\int_{\rr^d}\psi(x)dx\right).
$$
\end{theorem}

As an immediate application of this result we observe that, when
the first eigenvalue is positive, we have a decay estimate for the
solutions to the associated evolution problem
\eqref{eq.parabolica}.

\begin{theorem} \label{teo.decaimiento.intro}
Let $u(x,t)$ be the solution to \eqref{eq.parabolica} with $u(x,0)=u_0 (x) \in L^1 (\rr^d) \cap  L^\infty (\rr^d)$. Then, for any $r\in [1,\infty)$, the following hold:
\begin{itemize}
\item $\| u (\cdot, t) \|_{L^r (\rr^d)} \leq C t^{-\frac{r-1}{p-2}}$ for $2 < p < \infty$.

\bigskip

\item $\| u (\cdot, t) \|_{L^r (\rr^d)} \leq C e^{-\gamma t}$ for $1< p\leq 2$.
\end{itemize}
Here, $C>0$ and  $\gamma >0$  depend on $\| u_0 \|_{L^1(\rr^d)}$, $ \| u_0 \|_{L^{\infty}(\rr^d)}$, $r$, $p$ and $K$.
\end{theorem}

The extension of these results to the case of a general diffeomorphism $a$ of $\rr^d$ is left as an open problem.

Finally, we study the limit as $p \to \infty$ of $\lambda_{1,p}(\rr^d)$. Indeed, note that since the limit of the $L^p$-norm of a function is the $L^{\infty}$-norm of the function, the natural quantity to study is $\lim\limits_{p \to \infty } [\lambda_{1,p} (\rr^d)]^{1/p} $. So, the eigenvalue limit problem is the following:
\begin{equation*}\label{eq.infty.intro}
    \lambda_{1,\infty} (\rr^d)  =  \inf\left\{\frac{ \| u(x) - u(y)\|_{L^\infty (x,y \in supp (u); \, K(x,y) >0)} }{\| u  \|_{L^\infty (supp (u))}}: u \in L^\infty(\rr^d), \mbox{ compactly supported}\right\}.
\end{equation*}

For this limit problem we can state the next theorem.
\begin{theorem} \label{lim.p.0.intro}
Assume that the kernel is given by \eqref{forma.nucleo},
then
$$
\lim_{p \to \infty } [\lambda_{1,p} (\rr^d)]^{1/p}= \lambda_{1, \infty}(\rr^d) =0.
$$
\end{theorem}
\begin{remark}
Note that we have actually proved that $ \lim\limits_{p \to \infty}\lambda_{1,p} (\rr^d)=0$, so it is not possible to find an uniform lower bound for $p\ge1$.
\end{remark}

The paper is organized as follows: in Section~\ref{Sect.prelim} we show that the first eigenvalue of the  whole space can be approximated with the first one of a  sequence of expanding domains; in Section~\ref{main.results} we collect the proofs of the lower and upper bounds for the first eigenvalue and we prove Theorem~\ref{teo.cota.por.abajo}; in Section~\ref{Sect.estim.evol} we apply our previous results to obtain the decay estimates for the evolution problem and prove Theorem~\ref{teo.decaimiento.intro}; finally in Section \ref{Sect.infty} we estimate the first eigenvalue for $p=\infty$.

\section{The limit of the first eigenvalue in expanding domains}
\label{Sect.prelim}
\setcounter{equation}{0}

In this section we show that the first eigenvalue of our nonlocal
operator in the whole $\rr^N$ can be approximated by the first
eigenvalue in large domains. This result
is not used in the rest of the article but can be of independent
interest. The first eigenvalue in a bounded
domain,  $\lambda_{1,p} (\Omega)$, is defined as
$$
\lambda_{1,p} (\Omega)=\inf _{u\in L^p(\Omega)}\frac{\displaystyle
\int_{\rr^d}\int_{\rrd} K(x,y)|u(x)-u(y)|^pdxdy}{\displaystyle\int _{\Omega}|u|^p(x)dx}.
$$
Here we have extended $u$ to the whole $\rr^d$ by zero outside
$\Omega$.

As a preliminary step, we focus our attention in the case of balls
$B_R$ that are centered at the origin with radius $R$.

\begin{lemma}\label{asimptoticlimit}
Let $\lambda_{1,p}(\rr^d)$ be the first eigenvalue in the whole
space. Then
\begin{equation}\label{lambdalimit}
\lambda_{1,p} (\rr^d)=\lim _{R\rightarrow \infty} \lambda_{1,p} (B_R).
\end{equation}
\end{lemma}

\begin{proof} The proof is an adaptation of the one given in
\cite{nosotros.JDE} for the case $p=2$, nevertheless we provide a
sketch for the sake of completeness. First of all, we just remark
that $\lambda_{1,p} (\Omega)$ is decreasing with $\Omega$, that
is, if $\Omega_1\subset
\Omega_2$ then $\lambda_{1,p}(\Omega_1)\geq
\lambda_{1,p}(\Omega_2)$.
Then we deduce that there exists the limit
$$\lim _{R\rightarrow \infty} \lambda_{1,p}(B_{R})\geq 0.$$
Now, fix a function $u\in L^p(B_R)$. By the definition of
$\lambda_{1,p}(\rr^d)$, extending $u$ by zero outside $B_R$, we
get
$$\frac {\displaystyle
\int _{\rr^d}\int_{\rr^d} K(x,y)\vert \tilde u(x)-\tilde u(y)\vert^pdxdy}{\displaystyle\int _{B_R}\vert u\vert^p(x)dx}
\geq \lambda_{1,p}(\rr^d).$$
Taking the infimum in the right hand side over all functions $u\in
L^p(B_R)$ we obtain that for any $R>0$
$$
\lim _{R\rightarrow \infty}
\lambda_{1,p}(B_R)\geq \lambda_{1,p} (\rr^d).
$$

Now let be $\eps>0$. Then there exists $u_\eps\in L^p(\rr^d)$ such
that
$$
\lambda_{1,p}(\rr^d)+
\eps\geq \frac {\displaystyle\int_{\rr^d}\int_{\rrd} K(x,y)\vert u_\eps(x)-u_\eps(y)\vert^p dxdy}
{\displaystyle\int _{\rrd}\vert u_\eps\vert^p(x)dx}.
$$
We let $u_{\eps,R}$ defined by
$$u_{\eps,R}(x)=u_\eps (x)\chi_{B_R}(x),$$
and we observe that, when $R\to +\infty$, the following limits hold
$$
\int _{B_R} \vert u\vert^p_{\eps,R}(x)dx \longrightarrow \int_{\rr^d} \vert u_\eps\vert^p (x)dx
$$
and
$$
\int_{\rr^d}\int_{\rrd} K(x,y) \vert u_{\eps,R}(x)- u_{\eps,R}(y)\vert^p\, dx\, dy\longrightarrow
\int_{\rr^d}\int_{\rrd} K(x,y)\vert u_\eps(x)-u_\eps(y)\vert^p\, dx\, dy.
$$
Hence, using that $u_{\eps,R}$ vanishes outside the ball $B_R$ and
the definition of $\lambda_{1,p}(B_R)$ we get
$$\frac{\displaystyle\int _{\rr^d}\int_{\rrd} K(x,y) \vert u_{\eps,R}(x)- u_{\eps,R}(y)\vert^pdxdy}
{\displaystyle\int _{B_R} \vert u\vert^p_{\eps,R}(x)dx}
\geq \lambda_{1,p}(B_R).$$
Taking $R\rightarrow\infty $ we obtain
$$\frac{\displaystyle\int_{\rr^d}\int_{\rrd} K(x,y)\vert u_\eps(x)- u_\eps(y)\vert^pdxdy}
{\displaystyle\int _{\rr^d} \vert u_\eps\vert^p(x)dx}
 \geq \lim _{R\rightarrow \infty} \lambda_{1,p}(B_R).$$
Hence, for any $\eps>0$, we have $\lambda_{1,p}(\rr^d)+\eps\geq
\lim\limits_{R\rightarrow \infty}
\lambda_{1,p}(B_R) $. Thus
 $$\lambda_{1,p}(\rr^d)\geq \lim _{R\rightarrow \infty}\lambda_{1,p}(B_R)  ,$$
and then the proof of \eqref{lambdalimit} is finished.
\end{proof}

It is possible to extend this result to dilatations of a domain $\Omega$, such that $0 \in \Omega$.

\begin{theorem} \label{teo.conver} Let $\Omega \subset \rr^d$ be a domain such that $0 \in \Omega$. Then,
$$
\lim_{R \to \infty } \lambda_{1,p} (R\Omega) = \lambda_{1,p}
(\rr^d).
$$
\end{theorem}

\begin{proof} Let us
consider $B_{r_1} \subset \Omega \subset B_{r_2}$ then
$$
\lambda_{1,p} ( R B_{r_1}) \geq \lambda_{1,p} ( R \Omega) \geq
\lambda_{1,p} ( R B_{r_2}),
$$
and from the previous lemma we get that
$$
\lim_{R \to \infty} \lambda_{1,p} ( R B_{r_1}) = \lim_{R \to \infty} \lambda_{1,p} ( R B_{r_1}) =
\lambda_{1,p} (\rr^d). \qedhere
$$
\end{proof}

\section{Lower and upper bounds for the first eigenvalue}
\label{main.results}
\setcounter{equation}{0}

This section is mainly devoted to prove Theorem \ref{teo.cota.por.abajo}. Thus, up to the end of this section we will use the same notation as in the introduction, i.e., the kernel is given as in \eqref{forma.nucleo} and $1 \leq p < \infty$.

We first prove lower bounds for the first eigenvalue. To do that we use an elementary result.
We need to compute for a given $\eta \in (0,1)$ the optimal constant $\theta$ (the biggest one) such that
$    |a-b|^p \geq \eta a^p + \theta b^p$ for any $a,b\in \rr$.

\begin{lemma} \label{lema.clave.p} For any $0<\eta<1$ and $1\le p<\infty$ there exists $\theta(\eta,p)$ given by
$$
\theta (\eta,p) = \left\{ \begin{array}{ll}- \frac{\eta}{ (1-\eta^{\frac{1}{p-1}})^{p-1}}, \qquad & 1<p<\infty,\\
-\eta  \qquad & p=1, \end{array} \right.
$$
such that for all real  numbers $a$ and $b$ the following inequality holds:
\begin{equation*}
    |a-b|^p \geq \eta |a|^p + \theta |b|^p.
\end{equation*}
\end{lemma}

\begin{proof}
We just deal with $1< p< \infty$ since the case $p=1$ is simpler. Also, since $|a-b|\geq ||a|-|b||$ we can treat only the case of nonnegative numbers $a$ and $b$.
Taking $x=a/b$ we have to find $\theta \in \rr$ such that
$$
|x-1|^p - \eta x^p \geq \theta \qquad \mbox{ for all } x \ge 0.
$$
Therefore, the best value of $\theta$ is given by
$$
\theta (\eta) = \min_{x\geq 0}|x-1|^p - \eta x^p
$$
and we are left with the computation of this minimum.

First, let us consider $0\leq x \leq 1$, we have to compute the
minimum of
$$
f(x) = (1-x)^p - \eta x^p.
$$
Since $f$ is the sum of two decreasing functions then its minimum on the interval $[0,1]$ is attained at $x=1$:
$$
\min_{0\leq x \leq 1} f(x) = f(1) = - \eta.
$$

Now, for $x>1$ we have to find the minimum of
$
g(x) = (x-1)^p - \eta x^p .
$
In this case we compute the roots of its derivative
$$
g' (x) = p (x-1)^{p-1} - p \eta x^{p-1} =0
$$
and we obtain that the minimum if $g$ in the variable $x$ is attained at
$$
x(\eta) = \frac{1}{1-\eta^{\frac{1}{p-1}}}.
$$
Since $0 < \eta < 1$, we have
$$
- \frac{\eta}{ (1-\eta^{\frac{1}{p-1}})^{p-1}} < -\eta.
$$
Therefore, we obtain
$$
\theta (\eta,p) = - \frac{\eta}{ (1-\eta^{\frac{1}{p-1}})^{p-1}}
$$
as we wanted to show.
\end{proof}

We now give a lower estimate for the first eigenvalue $\lambda_{1,p} (\rr^d)$.
\begin{lemma} \label{lemma.geq} Given $1 \leq p < \infty$, then the following estimate holds
$$
\lambda_{1,p} (\rr^d)  \geq  2
\left( \int_{\rr^d} \psi (z ) \, dz \right) \left|
\displaystyle \frac{1}{|\det(A)|^{\frac{1}{p}}} -1
 \right|^p.
$$
\end{lemma}

\begin{proof} Let us first assume that $|\det(A)|\leq 1$ and $1 < p < \infty$. Using the symmetry of the kernel $K$,  given by \eqref{forma.nucleo}, we have
$$
\lambda_{1,p} (\rr^d)= 2 \inf\left\{\int_{\rr^d}\int_{\rrd} \psi (x -a(y))\, |u(x)-u(y)|^p\, dx\, dy : u\in L^p(\rr^d), \,  \| u \|_{L^p (\rr^d)}=1\right\}.
$$
Given $\varepsilon >0$, we can find a function $u \in L^p(\rr^d)$, with $\|u\|_{L^p(\rr^d)}=1$, such that
\[
\lambda_{1,p} (\rr^d) + \varepsilon  > 2 \int_{\rr^d}\int_{\rrd} \psi (x -a(y))\, |u(x)-u(y)|^p\, dx\, dy,
\]
and, applying Lemma \ref{lema.clave.p}, we obtain
$$
\begin{array}{l}
\displaystyle \lambda_{1,p} (\rr^d) + \varepsilon  >  2\int_{\rr^d}\int_{\rrd} \psi (x -a(y) )( \eta |u(x)|^p +\theta (\eta,p)|u(y)|^p )\, dx\,
dy
\\[8pt]
\displaystyle \qquad
=  2 \left(\displaystyle \eta \int_{\rr^d}\int_{\rrd} \psi (x - a(y) )  |u(x)|^p \, dx\, dy + \theta(\eta,p) \displaystyle{\int_{\rr^d}\int_{\rrd} \psi (x -a(y) ) |u(y)|^p \, dx\, dy} \right)
\\[8pt]
\displaystyle \qquad \displaystyle =  2
\left(
\displaystyle \frac{\eta}{|\det(A)|} + \theta(\eta,p)
\right)\int_{\rr^d} \psi (z ) \, dz.
\end{array}
$$
Since $\varepsilon >0$ was arbitrary, we conclude that
\[
\lambda_{1,p} (\rr^d) \ge 2 \left(\displaystyle \frac{\eta}{|\det(A)|} + \theta(\eta,p) \right)\int_{\rr^d} \psi (z ) \, dz.
\]
If we take
$$
\eta = (1-|\det(A)|^{\frac{1}{p}})^{p-1},
$$
then
$$
\theta (\eta,p) =  - \frac{\eta}{(1-\eta^{\frac{1}{p-1}})^{p-1}}= -\frac{(1-|\det(A)|)^{\frac{1}{p}})^{p-1}}{|\det(A)|^{\frac{p-1}{p}}}.
$$
Thus, we get
$$
\begin{array}{rl}
\displaystyle \lambda_{1,p} (\rr^d) & \displaystyle \geq  2
\left( \int_{\rr^d} \psi (z ) \, dz \right) \left(
\displaystyle \frac{(1-|\det(A)|)^{\frac{1}{p}})^{p-1}}{|\det(A)|}
 -  \frac{(1-|\det(A)|)^{\frac{1}{p}})^{p-1}}{|\det(A)|^{\frac{p-1}{p}}}
\right)
\\[8pt]
& \displaystyle =  2
\left( \int_{\rr^d} \psi (z ) \, dz \right) \left|
\displaystyle \frac{1}{|\det(A)|^{\frac{1}{p}}} -1
 \right|^p.
\end{array}
$$

The proof in the case  $|\det(A)|>1$ is analogous by
interchanging the roles of $x$ and $y$ and of $\eta$ and $\theta$. In fact, we arrive to
\[
\lambda_{1,p} (\rr^d) \ge 2 \left(\displaystyle \frac{\theta(\eta,p)}{|\det(A)|} + \eta \right)\int_{\rr^d} \psi (z ) \, dz.
\]
Now, we take
$$
\eta = (1-|\det(A)|^{\frac{-1}{p}})^{p-1} \quad \textrm{and} \quad
\theta (\eta,p) =  - \frac{\eta}{(1-\eta^{\frac{1}{p-1}})^{p-1}}
$$
to conclude.

For $p=1$ the proof is the same taking $\eta=1$ and $\theta=-1$.
\end{proof}

We now prove upper bounds for the first eigenvalue. Note that our proof is long and technical. Hence, we include here the details.

First, we prove a lemma that gives an upper bound of the first eigenvalue in terms of the integral of $\psi$ and an infimum  involving $a(x)$.

\begin{lemma}\label{cota.por.arriba} Given $1 \leq p < \infty$, then for any function $\phi$ with compact support, such that $\Vert \phi \Vert_{L^p(\rr^d)}=1$, we have
$$
\lambda_{1,p} (\rr^d) \leq  2
\left(\int_{\rr^d} \psi(z) \, dz\right) \int_{\rr^d} |\phi (x)-\phi (a(x))|^p\,dx.
$$
\end{lemma}
\begin{proof}
First of all, recall that we have
$$
\lambda_{1,p} (\rr^d)=\inf _{u\in L^p(\rr^d) }\frac{\displaystyle
\int_{\rr^d}\int_{\rrd} K(x,y)|u(x)-u(y)|^p\, dx\, dy}{\displaystyle\int _{\rr^d}|u|^p(x)dx}.
$$
Let us choose a nonnegative smooth function $\phi$, supported in the unit ball $B_1$, such
that $\Vert \phi \Vert_{L^p(\rr^d)}=1$, and consider a family of test functions $u(x) = \phi (x/R)$, for $R>0$. By a change of variable, we have

\begin{align*}
\lambda_{1,p}(\rr^d) &\leq \frac{\displaystyle{\int_{\rr^d}\int_{\rrd} K(x,y)\left|\phi \left(\frac{x}{R}\right)-\phi\left(\frac{y}{R}\right)\right|^p\, dx\, dy}}{\displaystyle{\int_{\rr^d}\phi^p\left(\frac{x}{R}\right)dx}}
\\[8pt]
& = R^d
\int_{\rr^d}\int_{\rr^d} K(Rx,Ry)\, |\phi (x)-\phi (y)|^p\, dx\, dy.
\end{align*}
Using  the definition of the kernel $K$ and the change of variable $z= Ry - a(Rx)$, we obtain
\begin{align*}
\lambda_{1,p}(\rr^d) &\leq  2 R^d
\int_{\rr^d}\int_{\rr^d} \psi(Ry - a(Rx)) \, |\phi (x)-\phi (y)|^p\, dx\, dy
\\[8pt]
& \displaystyle = 2
\int_{\rr^d}\int_{\rr^d} \psi(z)\left|\phi (x)-\phi \left(a(x) + \frac{z}{R}\right)\right|^p\, dz\,
dx.
\end{align*}

Thus, given $\varepsilon>0$, there exists a positive constant $C(\varepsilon)$ such that
$$
\begin{array}{rl}
\displaystyle \lambda_{1,p} (\rr^d) & \displaystyle \leq 2(1+\varepsilon)
\int_{\rr^d}\int_{\rr^d} \psi(z)|\phi (x)-\phi (a(x))|^p\, dz\,
dx \\ \\
& \qquad \displaystyle + C(\varepsilon)
\int_{\rr^d}\int_{\rr^d} \psi(z)\left|\phi (a(x))-\phi \left(a(x) + \frac{z}{R}\right)\right|^p\, dz\,
dx.
\end{array}
$$
Further, we have
$$
\begin{array}{rl}
\displaystyle \lambda_{1,p} (\rr^d) & \displaystyle \leq 2 (1+\varepsilon)
\left(\int_{\rr^d} \psi(z) \, dz\right)\int_{\rr^d} |\phi (x)-\phi (a(x))|^p\,
dx \\ \\
& \qquad \displaystyle + \frac{C(\varepsilon)}{R^p}
\int_{B_1} \psi(z) \int_{\rr^d} \int_0^1 \left\Vert \nabla \phi \left(a(x) + \frac{s z}{R}\right)\right\Vert^p \|z\|^p\, ds\, dx\, dz
\\ \\
& \displaystyle \leq 2(1+\varepsilon)
\left(\int_{\rr^d} \psi(z) \, dz\right)\int_{\rr^d} |\phi (x)-\phi (a(x))|^p\,
dx \\ \\
& \qquad \displaystyle + \frac{C(\varepsilon)}{R^p}
\int_{B_1} \psi(z)\, \|z\|^p\,  |\det(A)|^{-1}
\int_{\rr^d} \int_0^1 \left\Vert\nabla \phi \left(x + \frac{s z}{R}\right)\right\Vert^p \, ds\, dx\, dz.
\end{array}
$$
Therefore, letting $R\to \infty$ we obtain
$$
\lambda_{1,p} (\rr^d) \leq 2(1+\varepsilon)
\left(\int_{\rr^d} \psi(z) \, dz\right)\int_{\rr^d} |\phi (x)-\phi (a(x))|^p\,
dx
$$
and letting $\varepsilon \to 0$
$$
\lambda_{1,p} (\rr^d) \leq  2
\left(\int_{\rr^d} \psi(z) \, dz\right)\int_{\rr^d} |\phi (x)-\phi (a(x))|^p\,
dx,
$$
as we wanted to show.
\end{proof}

 According to Lemma \ref{lemma.geq}, we already know the lower bound for the first eigenvalue, and as a consequence, for any $\phi $ with compact support we have $$
\int_{\rr^d} |\phi (x)-\phi (a(x))|^p\,
dx \geq \left| \frac{1}{|\det(A)|^{\frac{1}{p}}} -1
 \right|^p.
$$
Our task is to construct a minimizing sequence. For future references, we prefer to be more specific and choose the minimizing sequence to be supported in the unit ball.

\begin{lemma}\label{por.arriba}  Let $1 \leq p < \infty$ and  $a: \mathbb{R}^d \to \mathbb{R}^d$
be an invertible linear map with $a(x)= Ax$.
There exists a sequence of nonnegative functions $\{ \phi_n \}_{n=1}^{\infty} \subset L^p(\rr^d) $, supported in the unit ball  with $\| \phi_n \|_{L^p(\rr^d)}=1$, such that
\begin{equation}
\label{mini}
\lim_{n \to \infty} \int_{\rr^d} |\phi_n (x)-\phi_n (a(x))|^p\,
dx = \left| 1 -  |\det(A)|^{-\frac{1}{p}} \right|^p.
\end{equation}
\end{lemma}

Our strategy for the proof of Lemma \ref{por.arriba} is  to construct  a sequence of functions as above for each Jordan block of $A$ and afterwards, using those functions, we define the desired sequence as a tensor product.

By the  Jordan's decomposition of $A$, there exist $C$ and  $J$ two
$d \times d$ invertible matrices with real entries such that $A=
CJC^{-1}$. Note that $J$ is defined by  Jordan blocks, i.e.,
\begin{equation}
\label{Jordan}
 J= \left (
\begin{array}{ccccccccc}
  J_1 (\lambda_1)&        &    &  &        &    &           \\
          &  \ddots &  & &        &    &        \\
          &         &         & J_r (\lambda_r) &        &    &      \\
    &        &    &      &       J_{r+1} (\alpha_1, \beta_1)&        &    &       \\
    &        &    &      &       &  \ddots &  &   \\
     &        &    &      &      &         &         & J_{r+ s} (\alpha_s, \beta_s)
\end{array}  \right ),
\end{equation}
with $|\lambda_1| \geq |\lambda_2| \geq \cdots \geq |\lambda_r| > 0$, $\alpha_i^2 + \beta^2_i \neq 0$, (recall that $A$ is invertible)
\begin{equation}
\label{realJordan}
 J_i(\lambda)= \left (
\begin{array}{ccccccccc}
  \lambda & 1       &    &       \\
          &  \ddots &  & 1  \\
          &         &         & \lambda
\end{array}  \right ), \quad i=1,\dots, r,
\end{equation}
and
\begin{equation}
\label{complexJordan}
 J_i(\alpha, \beta)= \left (
\begin{array}{ccccccccc}
  \mathbf{M} & \mathbf{I}       &    &       \\
          &  \ddots &  & \mathbf{I}  \\
          &         &         & \mathbf{M}
\end{array}  \right ), \quad i=r+1,\dots,r+s.
\end{equation}
Here $\lambda$, $\alpha$ and $\beta$ are real numbers,  $\mathbf{M}
=
\left (
\begin{array}{cc}
  \alpha & \beta       \\
    -\beta      &  \alpha
\end{array}  \right )
$ and $\mathbf{I} = \left (
\begin{array}{cc}
  1 & 0       \\
    0      &  1
\end{array}  \right )$.
Note that $J_i(\lambda_i)$ and  $J_i(\alpha_i,\beta_i)$ are the real Jordan blocks corresponding to  real and complex eigenvalues respectively.

Since the proof of Lemma \ref{por.arriba} is quite large we divide it in few steps  according to properties of the Jordan blocks.

First, we consider $J_i(\lambda_i)$ and $J_i(\alpha_i, \beta_i)$ when $| \lambda_i | >1$ and  $\alpha_i^2 + \beta_i^2 >1$. Indeed we write our construction in a more general context of expansive linear maps.
Recall that a linear map $a:\rr^d \to \rr^d$ with $a(x)= Ax$ is called expansive if all the (complex) eigenvalues of $A$ have absolute value bigger than one. Here and in what follows we denote by $\mid \cdot \mid_d$ the
Lebesgue measure of a set in $\rr^d$.

\begin{lemma}\label{minimizerexpansive} Let $1 \leq p < \infty$ and  $a: \mathbb{R}^d \to \mathbb{R}^d$
be an expansive linear map with $a(x)=Ax$. There exists a sequence of sets $\{ E_j \}_{j \in \nn_0} \subset B_1$ of positive measure such that
\begin{enumerate}
\item[(i)]    $a^{-1}(E_j) = E_{j+1}$, $|E_j|_d= |\det (A)|^{-j}|E_0|_d$ and
$
E_j \cap E_l= \emptyset,$ whenever $j,l \in \nn_0$ and $j \neq  l$.
\item[(ii)] If we choose $\sigma_n= |\det (A)|^{1/p} - \frac{1}{n} $ and
$$
\varphi_n (x)= \sum_{j=0}^{\infty} (\sigma_n)^j \chi_{E_j}(x), \qquad \textrm{for } n \geq 1,
$$
then the functions $\varphi_n$ are nonnegative and supported in the unit ball. Moreover, they belong to
$L^p(\rr^d)$ satisfying
\begin{eqnarray*}
\|\varphi_n \|^p_{L^p(\mathbb{R}^d)}=
\frac{|E_0|_d }{1- \sigma_n^p/ |\det (A)|} \qquad \mbox{ for } n\ge 1.
\end{eqnarray*}
\item[(iii)] The sequence   $\phi_n = \frac{\varphi_n}{\|\varphi_n \|_{L^p(\mathbb{R}^d)}}$, $n \geq 1$, satisfies \eqref{mini}.
\end{enumerate}
\end{lemma}

\begin{proof} We first construct our candidate to sequence of sets.
Since  $a$
is expansive, there exists $B \subset \mathbb{R}^d$ a ball with
center the origin such that $a^{-j}(B) \subset B_1$, $\forall j \in \nn_0$. Take the following sets
\begin{equation*}
\displaystyle F= \bigcup_{j=0}^{\infty}a^{-j}(B) \quad \textrm{and} \qquad E_l= a^{-l}(F) \setminus a^{-(l+1)}(F),
\quad \textrm{for } l \in \nn_0.
\end{equation*}

Let us first prove (i). By construction, we have that $F$ and $ E_l$ are subsets of the unit ball for all $\l \in \nn_0$. Moreover, we have that   $a^{-l}(E_0)= E_l$ and hence $|E_l|_d= |\det (A)|^{-l}|E_0|_d$.
Since $|\det (A)| >1$  and $|F|_d - |\det (A)|^{-1}|F|_d >0$ then $|E_0|_d>0$.
Hence we conclude that $E_l $ has positive measure for any $l\geq 0$
and, by the construction of the sets $E_l$, it follows that $E_l \cap E_j = \emptyset$  whenever $l \neq j$.
Also by construction, since $|\det A|>1$, $\varphi_n$ is nonnegative and supported in the unit ball for all $n\geq 1$. Moreover, $\varphi_n$ belongs to $L^p(\rr^d)$ and we can compute explicitly its norm
\begin{align*}
\|\varphi_n \|^p_{L^p(\mathbb{R}^d)} &= \sum_{j=0}^{\infty} \sigma_n^{jp} |E_j|_d =
|E_0|_d \sum_{j=0}^{\infty} \sigma_n^{jp} |\det (A)|^{-j} =|E_0|_d \sum_{j=0}^{\infty} \left(1 - \frac{|\det (A)|^{-1/p}}{n} \right)^{jp}\\
& =\frac{|E_0|_d }{1- \sigma_n^p/ |\det (A)|}.
\end{align*}

Finally, we prove (iii). For any positive number $n$ we have
\begin{eqnarray*}
\int_{\rr^d} |\varphi_n(x) - \varphi_n(a( x))|^p dx  &=& \int_{\rr^d}  \Big|
\sum_{j=0}^{\infty} \sigma_n^{j} \chi_{E_j} (x) -  \sum_{j=0}^{\infty} \sigma_n^{j} \chi_{E_j} (a(x)) \Big|^p dx\\  &=& \int_{\rr^d}  \Big|
\sum_{j=0 }^{\infty} \sigma_n^{j} \chi_{E_j} (x) -  \sum_{j=1}^{\infty} \sigma_n^{j-1} \chi_{E_{j-1}} (a(x)) \Big|^p dx.
\end{eqnarray*}
Since for any $j\geq 0$, $a^{-1}(E_j)= E_{j+1}$, we have $\chi_{E_{j-1}}(a(x))=\chi_{E_j}(x)$. Using also that  $E_j \cap E_l= \emptyset$ if $j \neq l$, we have
\begin{eqnarray*}
\int_{\rr^d} |\varphi_n(x) - \varphi_n(a( x))|^p dx & = & |E_0|_d + \int_{\rr^d}  \Big| \sum_{j=1}^{\infty} \sigma_n^{j-1} (\sigma_n -1) \chi_{E_j} (x) \Big |^p dx \\
& = &  |E_0|_d +  \sum_{j=1}^{\infty} \sigma_n^{(j-1)p}\, |\sigma_n -1|^p \, |E_j|_d.
\end{eqnarray*}
Using that $|E_j|_d= |\det (A)|^{-j} \, |E_0|_d$, it follows that
 \begin{eqnarray*}
\int_{\rr^d} |\varphi_n(x) - \varphi_n(a( x))|^p dx & = &  |E_0|_d +  |\sigma_n -1|^p \, |E_0|_d \, \sum_{j=1}^{\infty} \sigma_n^{(j-1)p} \,  |\det (A)|^{-j} \\
& = &  |E_0|_d +  |\sigma_n -1|^p \, |E_0|_d \, |\det (A)|^{-1} \,  \sum_{j=0}^{\infty} \sigma_n^{jp} \, |\det(A)|^{-j} \\
& = & |E_0|_d + |\sigma_n -1|^p \, |E_0|_d \, |\det (A)|^{-1} \, \dfrac{1}{1- \sigma_n^p/ |\det (A)|}.
\end{eqnarray*}
\medskip
Further, using the definition of the functions $\phi_n$ and since $\lim\limits_{n \to \infty} \|\varphi_n \|^p_{L^p(\mathbb{R}^d)} = \infty $,
$$
\lim_{n \to \infty }\int_{\rr^d} |\phi_n(x) - \phi_n(a( x))|^p dx = |\det (A)|^{-1}|\, |\det (A)|^{1/p} -1|^p= |1-|\det (A)|^{-1/p} |^p,
$$
which proves that \eqref{mini} holds for any expansive map $a$.
 \end{proof}

Next, we consider $J_i(\lambda_i)$ and $J_i(\alpha_i, \beta_i)$ when $0< | \lambda_i | <1$ and  $0< \alpha_i^2 + \beta_i^2 <1$. We write our construction in a more general context, where an invertible linear map $a$ satisfies that $a^{-1}$ is expansive. Using the same techniques as in the proof of Lemma \ref{minimizerexpansive} we obtain the following.
\begin{lemma}\label{minimizercontractive} Let $1 \leq p < \infty$ and  $a: \mathbb{R}^d \to \mathbb{R}^d$
be an invertible linear map with $a(x)=Ax$ such that $a^{-1}$ is expansive. Then  there exists a sequence  of sets $ \{ G_j \}_{j\in \nn_0} \subset B_1$ of positive measure such that
\begin{enumerate}
\item[(i)] $a(G_j) = G_{j+1}$, $|G_j|_d = |\det (A)|^{j}\, |G_0|_d$ and $G_j \cap G_l= \emptyset$, whenever $j,l \in \nn_0$ and $j \neq  l$.
\item[(ii)] If we choose $\gamma_n= |\det (A)|^{-1/p} - \frac{1}{n}$ and
$$
\varphi_n (x)= \sum_{j=0}^{\infty} (\gamma_n)^j \chi_{G_j}(x), \qquad  \mbox{ for } n \ge 1,
$$
then the functions $\varphi_n$ are nonnegative and supported in the unit ball. Moreover, they belong to
$L^p(\rr^d)$ satisfying
\begin{eqnarray*}
\|\varphi_n \|^p_{L^p(\mathbb{R}^d)}=\frac{1}{1-\gamma_n^p \, |\det (A)|}.
\end{eqnarray*}
\item[(iii)]  The sequence  $\phi_n = \frac{\varphi_n}{\|\varphi_n \|_{L^p(\mathbb{R}^d)}}$, $n \in \nn$, satisfies  \eqref{mini}.
\end{enumerate}
\end{lemma}

In the following three lemmas, we consider linear maps whose eigenvalues have absolute value equal to one.

First, we deal with the diagonalizable (in $\mathbb{C}$) case. In the next lemma, we deal with several real Jordan blocks simultaneously that correspond to the unitary eigenvalues $\lambda$ that have dimension one if $\lambda$ is real and two if $\lambda\in  \mathbb{C}\setminus \rr$.
\begin{lemma}\label{minimizeruni} Let $1 \leq p < \infty$ and $a: \mathbb{R}^d \to \mathbb{R}^d$, $a(x)=Ax$
with
\begin{equation}
\label{Jordandiagonalizable1}
 A= \left (
\begin{array}{ccccccccc}
 \mathbf{M}_1 (\lambda_1)&        &    &  &        &    &           \\
          &  \ddots &  & &        &    &        \\
          &         &         & \mathbf{M}_r (\lambda_r) &        &    &      \\
    &        &    &      &       \mathbf{M}_{r+1} (\alpha_1, \beta_1)&        &    &       \\
    &        &    &      &       &  \ddots &  &   \\
     &        &    &      &      &         &         & \mathbf{M}_{r+ s} (\alpha_s, \beta_s)
\end{array}  \right ),
\end{equation}
where $\lambda_i$, $\alpha_i$ and $\beta_i$ are real numbers such that $|\lambda_i| =1$ and $\alpha_i^2 + \beta^2_i =1$, $\mathbf{M}_i (\lambda_i)= \lambda_i$
and
\begin{equation*} \mathbf{M}_{i} (\alpha, \beta)
=
\left (
\begin{array}{ccccccccc}
  \alpha & \beta       \\
    -\beta      &  \alpha
\end{array}  \right ).
\end{equation*}
Then $\varphi= |B_1|^{-1/p}_d\chi_{B_1}$ satisfies
$$ \int_{\rr^d} |\varphi(x) - \varphi(a( x)) |^p dx =0.
$$
\end{lemma}
\begin{proof}
Since $a$ leaves invariant the unit ball, we have that $\varphi(a(x))= \varphi(x)$ and  the assertion
follows.
\end{proof}

In the next lemma, we deal with a non-diagonalizable Jordan block  corresponding to a real eigenvalue with modulus one.
\begin{lemma}\label{minimizeruni.44} Let $1 \leq p < \infty$, $\lambda \in \{ \pm 1 \}$ and $a: \mathbb{R}^d \to \mathbb{R}^d$, $a(x)=Ax$
be an invertible linear map such that the corresponding matrix
associated to the canonical basis is given by
\begin{equation}
\label{realJordan1}
A= \left (
\begin{array}{ccccccccc}
  \lambda & 1       &    &       \\
          &  \ddots &  & 1  \\
          &         &         & \lambda
\end{array}  \right ).
\end{equation}
Then there exists a sequence of bounded sets $ \{ E_j \}_{j=0}^{\infty} \subset \rr^d$ of positive measure such that
\begin{enumerate}
\item[(i)] $| E_j |_d = |E_0|_d$, $a(E_j) = E_{j+1}$  and $E_j \cap E_l= \emptyset,$ if $j,l \in \nn_0$, $j \neq  l$.
\item[(ii)] Given $n \in \nn$, if we choose
\[
\varphi_n (x)= \sum_{j=1}^{n}  \chi_{2^{-n}E_j}(x)
\]
then, the function $\varphi_n$ is  nonnegative, supported in the unit ball and moreover it belongs to
$L^p(\rr^d)$ with
\begin{eqnarray*}
\|\varphi_n \|^p_{L^p(\mathbb{R}^d)} = 2^{-nd}n
|E_0|_d .
\end{eqnarray*}
\item[(iii)] The sequence $\phi_n = \frac{\varphi_n}{\|\varphi_n \|_{L^p(\mathbb{R}^d)}}$, $n \in \nn$, satisfies
$$
\lim_{n \to \infty }\int_{\rr^d} |\phi_n(x) - \phi_n(a( x))|^p dx =0.
$$
\end{enumerate}
  \end{lemma}
\begin{proof}  Let us  start constructing a sequence of sets which satisfies (i). Given $j \in \nn_0$, $a^{j}(\mathbf{p})^t=(A^j\mathbf{p})^t= (\lambda^j + j \lambda^{j-1}, \lambda^j, 0, \dots, 0 )^t$ where $\mathbf{p}=(1, 1, 0, \dots, 0)\in \rr^d$. Observe that $a^j(\mathbf{p}) \neq a^l(\mathbf{p})$ if $l,j \in \mathbb{N}$ and $j \neq l$. Indeed $| a^j(\mathbf{p}) - a^l(\mathbf{p}) | \geq 1$ if $j \neq l$. Here and in what follows, we use the notation $|\cdot |$ ffor the Euclidean norm in $\mathbb{R}^d$. Thus,
$a^j(B_{1/4}(\mathbf{p})) \cap a^l(B_{1/4}(\mathbf{p})) = \emptyset$  if $j \neq l$, where $B_{1/4}(\mathbf{p})$ is the ball with center the point $\mathbf{p}$ and radius $1/4$.

For $j \in \nn_0$,  we take $E_j= a^{j}(2^{-3}B_{1/4}(\mathbf{p}))= a^{j}(B_{2^{-5}}(2^{-3}\mathbf{p}))$.
Then the sequence $\{ E_j \}_{j \in \nn}$ satisfies (i).

We now prove see (ii).  Given $n \in \mathbb{N}$, we check  that the function $\varphi_n$ is supported in the unit ball. If
$x$ is in the support of $\varphi_n$ then there exists $j \in
\{0,1,\dots, n\}$ such that $| x - a^j(2^{-n-3}\mathbf{p}) | \leq 2^{-n-5}$
and we have
$$ | x  |  \leq | x - a^j(2^{-n-3}\mathbf{p}) | + |  a^j(2^{-n-3} \mathbf{p}) | \leq  2^{-n-5} + 2^{-n-3}((1 + n )^2 + 1)^{1/2}
 \leq  2^{-n} n< 1.
$$
Further, using (i)
$$
\|\phi_n\|^p_{L^p(\mathbb{R}^d)}= \sum_{j=1}^{n}  |2^{-n}E_j|_d
= 2^{-nd}n|E_0|_d.
$$

We now prove (iii). Given $n \in \nn$, by (i) and (ii) we have
\begin{eqnarray*}
\int_{\rr^d} |\varphi_n(x) - \varphi_n(a( x))|^p dx &=&  \int_{\rr^d} \Big|\sum_{j=1}^n \chi_{2^{-n}E_j}(x) - \sum_{j=1}^n \chi_{2^{-n}E_j}(a(x))\Big|^p dx \\ &=&  2^{-nd} |E_n |_d + 2^{-nd} |E_0 |_d= {2^{-nd+1} |E_0|_d}.
\end{eqnarray*}
Thus, $\phi_n$ satisfies
\begin{eqnarray*}
\lim_{n \to \infty} \int_{\rr^d} |\phi_n(x) - \phi_n(a( x))|^p dx = \lim_{n \to \infty}  \frac{2}{ n} =0,
\end{eqnarray*}
as we wanted to prove.
\end{proof}

In the next lemma, we deal with a non-diagonalizable Jordan block  corresponding to a complex eigenvalue with modulus one.
\begin{lemma}\label{minimizeruni.444} Let $1 \leq p < \infty$,  $n \in \nn$ and $a: \mathbb{R}^d \to \mathbb{R}^d$, $a(x)=Ax$
with
\begin{equation} \label{complexJordan1}
A= \left (
\begin{array}{ccccccccc}
  \mathbf{M} & \mathbf{I}       &    &       \\
          &  \ddots &  & \mathbf{I}  \\
          &         &         & \mathbf{M}
\end{array}  \right ),
\end{equation}
where  $\theta \in \mathbb{R} \setminus \{ \pi k\}_{k \in \mathbb{Z}}$, $\mathbf{M}
= \left (
\begin{array}{ccccccccc}
  \cos \theta & \sin \theta       \\
    -\sin \theta      &  \cos \theta
\end{array}  \right )
$ and $\mathbf{I} = \left (
\begin{array}{ccccccccc}
  1 & 0       \\
    0      &  1
\end{array}  \right )$.
For any positive integer $n$  there exists a finite sequence of bounded sets $\{ E_j^{(n)} \}_{j=0}^{n} \subset \rr^d$ of positive measure such that
\begin{enumerate}
\item[(i)] $| E_j^{(n)} |_d = |E_0^{(n)}|_d
$,  $a(E_j^{(n)}) = E_{j+1}^{(n)}$  and
$
E_j^{(n)} \cap E_l^{(n)}= \emptyset,$ if $j,l \in \{0,1,\dots, n\}$, $j \neq  l$.
\item[(ii)] For any positive integer $n$, $\varphi_n$ defined by
$$
\varphi_n (x)= \sum_{j=1}^{n}  \chi_{2^{-n}E_j^{(n)}}(x), \qquad n \in \nn.
$$
is  nonnegative and  supported in the unit ball. Moreover it belongs to
$L^p(\rr^d)$ and
\begin{eqnarray*}
\|\varphi_n \|^p_{L^p(\mathbb{R}^d)} = 2^{-nd}n
|E_0^{(n)}|_d .
\end{eqnarray*}
\item[(iii)] The sequence  $\phi_n = \frac{\varphi_n}{\|\varphi_n \|_{L^p(\mathbb{R}^d)}}$ verifies
$$
\lim_{n \to \infty }\int_{\rr^d} |\phi_n(x) - \phi_n(a( x))|^p dx =0.
$$
\end{enumerate}
  \end{lemma}
\begin{proof}  Fix $n \in \nn$.
Let us start the proof constructing $\{ E_j^{(n)} \}_{j=0}^{n}$ a sequence of bounded sets  satisfying (i). Denote  $\mathbf{q} =(1,1,1,1,0\dots 0)$. This selection can be done only if we have at least two blocks. Note that $d\geq 4$.
 Then, for any $j \in \nn $,
$$
 a^{j}(\mathbf{q}) = \left (
\begin{array}{ccccccccc}
 \cos (j\theta) +\sin (j\theta) + j \cos ((j-1)\theta) +j\sin ((j-1)\theta)
 \\ \cos (j\theta) -\sin (j\theta) + j \cos ((j-1)\theta) -j\sin ((j-1)\theta)      \\
         \cos (j\theta) +\sin (j\theta)  \\
         \cos (j\theta) -\sin (j\theta) \\
         0 \\
         \cdots \\
         0
\end{array}  \right ).
$$
Observe that
 $a^{j}(\mathbf{q}) \neq \mathbf{q}$ if $j \in \nn$. So $a^j(\mathbf{q}) \neq
a^l(\mathbf{q})$ if $l,j \in  \{ 0, \dots, n \}$ and $j \neq l$.
Thus,  by the continuity of the linear map $a$,  there exists $ 0< r=r(n) < 1$ and  $B_{r(n)}(\mathbf{q}) \subset
\rr^d$, a ball with the center at  the point $\mathbf{q}$ and radius $r$,
such that $a^j(B_{r(n)}(\mathbf{q})) \cap a^l(B_{r(n)}(\mathbf{q})) = \emptyset$  if
$j,l \in \{ 0,1, \dots, n \}$, $j \neq l$.

For $j \in \{0,1,\cdots \} $,  we take $E_j= a^{j}(2^{-4}B_{r(n)}(\mathbf{q}))= a^{j}(B_{2^{-4}r(n)}(2^{-4}\mathbf{q}))$,
then the sequence of set $\{ E_j^{(n)} \}_{j =0}^{\infty}$ satisfies (i).

We now prove (ii).  Given $n \in \mathbb{N}$, we check  that the function $\varphi_n$,
with the sets $E_j$ defined in the previous case, is supported in the unit ball. If $x$ is in the support of
$\varphi_n$ then there exists $j \in \{1,\dots, n\}$ such that $|
x - 2^{-n}a^j(2^{-4}\mathbf{q}) | \leq 2^{-n-4}$ and we have
\begin{eqnarray*}
 | x  | & \leq& | x - 2^{-n}a^j(2^{-4}\mathbf{q}) | +
|  2^{-n}a^j(2^{-4}\mathbf{q}) | \\
 & \leq & 2^{-n-4} + 2^{-n-4}(2(2 + 2j)^2 + 2^3 )^{1/2}
 \leq  2^{-n}n  < 1.
 \end{eqnarray*}
Further, using (i) we obtain
$$
\|\varphi_n\|^p_{L^p(\mathbb{R}^d)}= \sum_{j=0}^{n}  |2^{-n}E_j^{(n)}|_d = 2^{-nd}n|E_0^{(n)}|_d.
$$

We now prove (iii). Let  $n \in \nn$, by (i) and (ii) we have then
\begin{eqnarray*}
\int_{\rr^d} |\phi_n(x) - \phi_n(a( x))|^p dx &=& \frac{1}{\|\varphi_n\|^p_{L^p(\mathbb{R}^d)}} \int_{\rr^d} |\sum_{j=0}^n \chi_{a^j(2^{-n}B)}(x) - \sum_{j=0}^n \chi_{a^j(2^{-n})}(a(x))|^p dx \\ &=&  \frac{1}{\|\varphi_n\|^p_{L^p(\mathbb{R}^d)}}2^{-nd+1 }|E_0^{(n)}|_d
\end{eqnarray*}
and the conclusion  follows.
\end{proof}

For the proof of Lemma \ref{por.arriba}, we also need the following three lemmas where the  matrix of the linear map $a$ contains several Jordan blocks.

In the next lemma we combine two blocks (not necessarily Jordan blocks), one of then is expansive and the other is contractive.
\begin{lemma} \label{lemma:expansivecontractive}
Let $1 \leq p < \infty$, $d, d_1, d_2 \in \nn$ such that $d = d_1 + d_2$. Moreover let  $a: \mathbb{R}^d \to \mathbb{R}^d$, $a(x)=Ax$
be an invertible linear map such that the corresponding matrix
associated to the canonical basis is given by
$$
A= \left (
\begin{array}{ccccccccc}
  P & 0       \\
    0      &  Q
\end{array}  \right ),
$$
where $P$ is a $d_1 \times d_1$ expansive matrix and $Q^{-1}$ is a $d_2 \times d_2$ expansive matrix. Let $\{ \varphi_n \}_{n \in \nn} $ be the sequence of functions   defined on $\rr^{d_1}$ as in Lemma \ref{minimizerexpansive} when the dilation is given by $P$. Moreover, let $\{ \theta_n \}_{n \in \nn} $ be the sequence of functions   defined on $\rr^{d_2}$ as in Lemma \ref{minimizercontractive} when the dilation is given by $Q$. Then
$$
\lim_{n \to \infty }\int_{\rr^{d_1}} \int_{\rr^{d_2}}|\phi_n(x)\varrho_n(y) - \phi_n(P( x)) \varrho_n(Q(y))|^p \, dx dy =\left| 1- |\det(A)|^{-1/p} \right|^p,
$$
 where $\phi_n = \frac{\varphi_n}{\|\varphi_n \|_{L^p(\mathbb{R}^{d_1})}}$ and  $ \varrho_n = \frac{\theta_n}{\|\theta_n \|_{L^p(\mathbb{R}^{d_2})}}$
\end{lemma}
\begin{remark}
Note that here the product $\psi_n\varphi_n$ does not have the support in the unit ball. The support is in $|(x,y)|\leq 2$ and a change of variable gives us the right support.
\end{remark}

\begin{proof}
For $n \in \mathbb{N}$, we choose $\varphi_n (x)= \sum_{j=0}^{\infty} (\sigma_n)^j \chi_{E_j}(x)$, where $\sigma_n= |\det (P)|^{1/p} - \frac{1}{n}$ and  $ \{ E_j\}_{j=0}^{\infty}$  as  in Lemma \ref{minimizerexpansive}. Also  $\theta_n (y)= \sum_{j=0}^{\infty} (\gamma_n)^j \chi_{G_j}(y)$, where $\gamma_n= |\det (Q)|^{-1/p} - \frac{1}{n}$ and  $ \{ G_j\}_{j=-1}^{\infty}$ as in Lemma \ref{minimizercontractive}.

Using $P^{-1}(E_j)= E_{j+1}$ and $Q(G_j)= G_{j+1}$, we have
$$
\begin{array}{l}
 \displaystyle I_n  : = \int_{\rr^{d_1}} \int_{\rr^{d_2}} \Big|\varphi_n(x)\theta_n(y) - \varphi_n(P( x)) \theta_n(Q(y))\Big|^p \, dx dy \\[8pt]
 \displaystyle = \int_{\rr^{d_1}} \int_{\rr^{d_2}} \Big| \Big(\sum_{j=0}^{\infty} \sigma_n^{j} \chi_{E_j} (x)\Big) \Big(\sum_{l=0}^{\infty} \gamma_n^l \chi_{G_l}(y)\Big)  -  \Big(\sum_{j=1}^{\infty} \sigma_n^{j-1} \chi_{E_j} (x)\Big) \big(\sum_{l=-1}^{\infty} \gamma_n^{l+1} \chi_{G_l}(y)\Big) \Big|^p  dx dy \\[8pt]
 \displaystyle = \int_{\rr^{d_1}} \int_{\rr^{d_2}}  \Big| \varphi_n(x)\theta (y) -\Big( \frac {\varphi_n(x)}{\sigma _n} -\chi_{E_0}(x)\Big)
 \Big(\gamma_n\theta _n(y)-\chi_{G^{-1}}(y)\Big)
  \Big|^p  dxdy\\[8pt]
  \displaystyle = \int_{\rr^{d_1}} \int_{\rr^{d_2}}  \Big| \varphi_n(x)\theta_n (y) \big(1-\frac {\gamma_n}{\sigma_n}\big)
  +\gamma_n\chi_{E_0}(x) \theta_n(y)-\Big(\frac{ \varphi_n(x)}{\sigma_n}-\chi_{E_0}(x)\Big)
  \chi_{G_{-1}}(y) \Big|^p  dxdy
\end{array}
$$
where $G_{-1}= Q^{-1}G_0$.
Since $G_l \cap G_j = \emptyset$,   if $l \neq j$, $G_{-1}$ is also disjoint from $G_l$, $l \neq -1$, and $|G_{-1}|_{d_2}=|\det Q|^{-1} |G_0|_{d_2}$ then
\begin{align*}
I_n=& \int_{\rr^{d_1}} \int_{\rr^{d_2}}  \Big| \varphi_n(x)\theta_n (y) \big(1-\frac {\gamma_n}{\sigma_n}\big)
  +\gamma_n\chi_{E_0}(x) \theta_n(y) \Big|^p  dxdy\\
  &\qquad +
   |G_{-1}|_d  \int_{\rr^{d_1}}\Big| \frac{ \varphi_n(x)}{\sigma_n}-\chi_{E_0}(x)
  \Big|^p  dx\\
  :=& I_{n1}+I_{n2}.
\end{align*}
Since $\sigma_n$ is bounded and $\|\theta_n\|_{L^p(\rr^{d_2})}\rightarrow \infty$ we have that
$$
\lim _{n\rightarrow \infty}\frac{I_{n2}}{\|\varphi_n\|^p_{L^p(\rr^{d_1})}\|\theta_n\|^p_{L^p(\rr^{d_2})}}=0.
$$
For the first term we observe that, since $\|\varphi_n\|_{L^p(\rr^{d_1})}\rightarrow \infty$ we have that
$$
\lim _{n\rightarrow \infty} \dfrac{\displaystyle{\int_{\rr^{d_1}} \int_{\rr^{d_2}}  |\gamma_n\chi_{E_0}(x) \theta_n(y)|^p \, dxdy}}{\|\varphi_n\|^p_{L^p(\rr^{d_1})}\|\theta_n\|^p_{L^p(\rr^{d_2})}}=0.
$$
Then
\begin{eqnarray*}
 & & \limsup_{n\rightarrow \infty}\frac{I^{1/p}_{n1}}{\|\varphi_n\|_{L^p(\rr^{d_1})}\|\theta_n\|_{L^p(\rr^{d_2})}} \\ & & \leq
 \limsup_{n\rightarrow \infty}\frac{1}{\|\varphi_n\|_{L^p(\rr^{d_1})}\|\theta_n\|_{L^p(\rr^{d_2})}}  \left( \int_{\rr^{d_1}} \int_{\rr^{d_2}}  \Big| \varphi_n(x)\theta_n (y) \big(1-\frac {\gamma_n}{\sigma_n}\big)
  \Big|^p  dxdy \right)^{1/p} \\ \qquad & &  \qquad + \lim_{n\rightarrow \infty}\frac{1}{\|\varphi_n\|_{L^p(\rr^{d_1})}\|\theta_n\|_{L^p(\rr^{d_2})}} \left( \int_{\rr^{d_1}} \int_{\rr^{d_2}}  \Big| \gamma_n\chi_{E_0}(x) \theta_n(y) \Big|^p  dxdy \right)^{1/p}.
\end{eqnarray*}
Thus, using that $(1-\gamma_n/\sigma_n) \rightarrow (1-|\det A|^{-1/p})$ we obtain that
$$\limsup_{n\rightarrow \infty}\frac{I^{1/p}_{n1}}{\|\varphi_n\|^p_{L^p(\rr^{d_1})}\|\theta_n\|^p_{L^p(\rr^{d_2})}} \leq  1-|\det A|^{-1/p}.
$$
Moreover,
\begin{eqnarray*}
 & & \liminf_{n\rightarrow \infty}\frac{I^{1/p}_{n1}}{\|\varphi_n\|_{L^p(\rr^{d_1})}\|\theta_n\|_{L^p(\rr^{d_2})}} \\ & &\geq
 \liminf_{n\rightarrow \infty}\frac{1}{\|\varphi_n\|_{L^p(\rr^{d_1})}\|\theta_n\|_{L^p(\rr^{d_2})}}  \left( \int_{\rr^{d_1}} \int_{\rr^{d_2}}  \Big| \varphi_n(x)\theta_n (y) \big(1-\frac {\gamma_n}{\sigma_n}\big)
  \Big|^p  dxdy \right)^{1/p} \\ \qquad & & \qquad  - \lim_{n\rightarrow \infty}\frac{1}{\|\varphi_n\|_{L^p(\rr^{d_1})}\|\theta_n\|_{L^p(\rr^{d_2})}} \left( \int_{\rr^{d_1}} \int_{\rr^{d_2}}  \Big| \gamma_n\chi_{E_0}(x) \theta_n(y) \Big|^p  dxdy \right)^{1/p}.
\end{eqnarray*}
As before we get
$$\liminf_{n\rightarrow \infty}\frac{I^{1/p}_{n1}}{\|\varphi_n\|^p_{L^p(\rr^{d_1})}\|\theta_n\|^p_{L^p(\rr^{d_2})}} \geq  1-|\det A|^{-1/p}.
$$
and the proof is finished.
\end{proof}
In the next lemma we combine two blocks (not necessarily Jordan blocks), one of them correspond to  non-diagonalizable Jordan blocks with the eigenvalues of modulus one, while the other is an arbitrary invertible matrix. Indeed this lemma shows that we can simplify our computations neglecting the first block.
\begin{lemma} \label{lemma:1otro}
Let $1 \leq p < \infty$, $d, d_1, d_2 \in \nn$ such that $d = d_1 + d_2$. Moreover let  $a: \mathbb{R}^d \to \mathbb{R}^d$, $a(x)=Ax$
be an invertible linear map such that the corresponding matrix
associated to the canonical basis is given by
$$
 A= \left (
\begin{array}{ccccccccc}
  P & 0       \\
    0      &  Q
\end{array}  \right ),
$$
where $P$ is a $d_1 \times d_1$  matrix as \eqref{realJordan1} or \eqref{complexJordan1}  and $Q$ is a $d_2 \times d_2$ invertible matrix. Let $\{ \varphi_n \}_{n \in \nn} $ be the sequence of functions   defined on $\rr^{d_1}$ as in  Lemma \ref{minimizeruni.44} or Lemma \ref{minimizeruni.444} when the dilation is given by $P$. Moreover, let $\{ \theta_n \}_{n \in \nn} $ be a sequence of functions on $\rr^{d_2}$ such that $\|
\theta_n \|^p_{L^p(\rr^{d_2})}=1$. If $\phi_n = \frac{\varphi_n}{\|\varphi_n \|_{L^p(\mathbb{R}^{d_1})}}$, then
$$
\lim_{n \to \infty }\int_{\rr^{d_1}} \int_{\rr^{d_2}}|\phi_n(x)\theta_n(y) - \phi_n(P( x)) \theta_n(Q(y))|^p \, dx dy = \lim_{n \to \infty } \int_{\rr^{d_2}}|\theta_n(y) -  \theta_n(Q(y))|^p \,  dy
$$
if both previous limits exist.
\end{lemma}
\begin{proof}
For any $n \in \mathbb{N}$, $\varphi_n$ is given by  $\varphi_n(x)= \sum_{j=1}^{n} \ \chi_{2^{-n}E_j^{(n)}}(x)$ where  the sets  $ \{ E_j^{(n)}\}_{j=0}^{\infty}$  satisfy (i) in  Lemma \ref{minimizeruni.44} or in Lemma \ref{minimizeruni.444}. Thus, we have
\begin{eqnarray*}
 I_n &:= & \int_{\mathbb{R}^{d_1} } \int_{\mathbb{R}^{d_2} } | \phi_n(x)\theta_n(y) -\phi_n(Px)\theta_n(Q y)|^p\,
dxdy \\  & = & \frac{1}{\|
\varphi_n \|^p_{L^p(\rr^{d_1})}} \int_{\mathbb{R}^{d_1} } \int_{\mathbb{R}^{d_2} } | (\sum_{j=1}^{n} \ \chi_{2^{-n}E_j^{(n)}}(x))\theta_n(y) -(\sum_{j=1}^{n} \ \chi_{2^{-n}E_j^{(n)}}(Px))\theta_n(Q y)|^p\,
dxdy \nonumber \\  & = & \frac{1}{\|
\varphi_n \|^p_{L^p(\rr^{d_1})}} \int_{\mathbb{R}^{d_1} } \int_{\mathbb{R}^{d_2} } | (\sum_{j=1}^{n} \ \chi_{2^{-n}E_j^{(n)}}(x))\theta_n(y) -(\sum_{j=0}^{n-1} \ \chi_{2^{-n}E_j^{(n)}}(x))\theta_n(Q y)|^p\,
dxdy. \nonumber
  \end{eqnarray*}
Since $E_j^{(n)}  \cap E_l^{(n)} = \emptyset$ if $j \neq l$, then
\begin{equation} \label{aaaa1}
 I_n   \leq   \int_{\mathbb{R}^{d_2} } | \theta_n(y) -\theta_n(Q y)|^p\,
dy +  \frac{2^{-nd+1} }{\|
\varphi_n \|^p_{L^p(\rr^{d_1})}}   |E_0|_{d_1} (1 + |\det Q|^{-1}).
  \end{equation}
  Moreover,
  \begin{equation} \label{aaaa11}
 I_n   \geq   \int_{\mathbb{R}^{d_2} } | \theta_n(y) -\theta_n(Q y)|^p\,
dy -  \frac{2^{-nd+1} }{\|
\varphi_n \|^p_{L^p(\rr^{d_1})}}   |E_0|_{d_1} (1 + |\det Q|^{-1}).
  \end{equation}
  Finally, since $\|
\varphi_n \|^p_{L^p(\rr^{d_1})}= 2^{-nd} n |E_0|_{d_1}$, the last terms in  \eqref{aaaa1} and \eqref{aaaa11} go to zero and  the conclusion follows.
\end{proof}

The following lemma is analogous to the previous one but the first block is related with diagonalizable Jordan blocks with the eigenvalues of modulus one
\begin{lemma} \label{lemma:1111}
Let $1 \leq p < \infty$, $d, d_1, d_2 \in \nn$ such that $d = d_1 + d_2$. Moreover let  $a: \mathbb{R}^d \to \mathbb{R}^d$, $a(x)=Ax$
with
$$
 A= \left (
\begin{array}{ccccccccc}
  P & 0       \\
    0      &  Q
\end{array}  \right ),
$$
where $P$ is a $d_1 \times d_1$ invertible matrix as in \eqref{Jordandiagonalizable1}   and $Q$ is a $d_2 \times d_2$ invertible matrix. Let $\phi  = |B_1|^{-1/p} \chi_{B_1}$ where $B_1$ denotes the unitary ball on $\rr^{d_1}$. Moreover, let $ \theta $ be a function on $\rr^{d_2}$ such that $\|
\theta \|^p_{L^p(\rr^{d_2})}=1$. Then
$$
\int_{\rr^{d_1}} \int_{\rr^{d_2}}|\phi(x)\theta(y) - \phi(P( x)) \theta(Q(y))|^p \, dx dy =   \int_{\rr^{d_2}}|\theta(y) -  \theta(Q(y))|^p \,  dy.
$$
\end{lemma}
\begin{proof}
 Since $\phi(Px)= \phi(x)$, the statement follows.
 \end{proof}
We are ready to prove Lemma \ref{por.arriba}.
\begin{proof}[Proof of Lemma \ref{por.arriba}]
We write the matrix $A$ using a Jordan decomposition \eqref{Jordan} with  $d_i \times d_i$ Jordan blocks $J_i$ as in
(\ref{realJordan}) or (\ref{complexJordan}).

If $J_i$ is expansive, we construct a sequence of functions  $\{ \phi_n^{(i)} \}_{n=1}^{\infty} $ defined on variables $(x_1^{(i)},\cdots, x_{d_i}^{(i)})$ as in Lemma \ref{minimizerexpansive}. If $J_i^{-1}$ is expansive then the sequence of functions is constructed as in Lemma \ref{minimizercontractive}. If $J_i$ is diagonalizable with the absolute values of the eigenvalues equals to $1$,  we take $\phi_n (x_1^{(i)},\cdots, x_{d_i}^{(i)}) = |B_1|^{1/p} \chi_{B_1}(x_1^{(i)},\cdots, x_{d_i}^{(i)})$ where $B_1$ is the unit ball on $\rr^{d_i}$. When  $J_i$ is not diagonalizable and the real or complex eigenvalues with absolute value $1$, then the sequence of functions is constructed as in Lemma \ref{minimizeruni.44} or Lemma \ref{minimizeruni.444}.

For $n \in \mathbb{N}$, we now choose
$$
\Psi_n(x_1^{(1)}, \dots, x_{d_1}^{(1)}, \dots \dots,
x_1^{(r+s)}, \dots, x_{d_{(r+s)}}^{(r+s)})= \prod_{i=1}^{r+s} \phi_n^{(i)}
(x_1^{(i)}, \dots, x_{d_i}^{(i)})
$$
 and
$$
\Phi_n(x)= d^{-d/2p}\| C^{-1} \|^{-d/p} | \det C |^{-1/p}
\Psi_n (d^{-1/2}\| C^{-1} \|^{-1} C^{-1} x),
$$
where  $\| C^{-1} \|$ denotes the norm of $C^{-1}$ as operator on
$\rr^d$. Observe that  $\Phi_n$ is supported in $B_{1}$, $\|
\Phi_n \|_{L^p(\rr^d)}=1$ and $ \Phi_n(x) $ is nonnegative.

We now check that $\{  \Phi_n \}$ defined above satisfies
\begin{equation}
\label{miniPhi}
\lim_{n \to \infty} \int_{\rr^d} |\Phi_n (x)-\Phi_n (a(x))|^p\,
dx = \left| 1 -  |\det(A)|^{-\frac{1}{p}} \right|^p.
\end{equation}

After the change of variable $d^{-1/2} \| C^{-1}
\|^{-1}C^{-1}x=y$,   we have
$$
\begin{array}{l}
\displaystyle \lim_{n \to \infty}\int_{\mathbb{R}^{d} }
\int_{\rr^d} |\Phi_n (x)-\Phi_n (a(x))|^p\,
dx \\[8pt]
\displaystyle  =  d^{-d/2}\| C^{-1} \|^{-d} | \det C |^{-1}
  \\[8pt] \displaystyle
\qquad   \qquad \times
\lim_{n \to \infty}\int_{\mathbb{R}^{d} } |\Psi_n (d^{-1/2}\| C^{-1} \|^{-1} C^{-1} x)-\Psi_n (d^{-1/2}\| C^{-1} \|^{-1} C^{-1} CJC^{-1}x)|^p\,
dx \\[8pt]
\displaystyle  =   \lim_{n \to \infty}\int_{\mathbb{R}^{d} } |\Psi_n (y)-\Psi_n (Jy)|^p\,
dy  \\[8pt]
\displaystyle  =  \lim_{n \to \infty}\int_{\mathbb{R}^{d_1} } \cdots \int_{\mathbb{R}^{d_{r+s}} } \left|\prod_{i=1}^{r+s}  \phi_n^{(i)} (z^{(i)})-\prod_{i=1}^{r+s} \phi_n^{(i)} (J_iz^{(i)})\right|^p\,
dz^{(1)}\cdots dz^{(r+s)}
 \end{array}
$$
where $z^{(i)}= (x_1^{(i)},\cdots, x_{d_i}^{(i)})$.

  We study the following cases.

 {\bf Case I.} {\em The matrix $A$ contains a single real Jordan block $J(\lambda)$.}

  If  $|\lambda|\neq 1$  then by Lemma \ref{minimizerexpansive} and Lemma \ref{minimizercontractive}, \eqref{miniPhi} holds. If $\lambda$ has the absolute value  equals to one, then the  conclusion  follows by Lemma \ref{minimizeruni}, Lemma \ref{minimizeruni.44} or Lemma \ref{minimizeruni.444}.

 {\bf Case II.} {\em  The matrix $A$  contains several real Jordan blocks.}

  Assume there exists a Jordan block   denoted $J_1(\lambda_1)$ that corresponds to the eigenvalue $\lambda_1$ with $|\lambda_1|=1$.  Then by   Lemma \ref{lemma:1otro} and Lemma \ref{lemma:1111},
   \begin{eqnarray*}
 & & \lim_{n \to \infty } \int_{\mathbb{R}^{d_1} } \cdots \int_{\mathbb{R}^{d_{r+s}} } \left|\prod_{i=1}^{r+s}  \phi_n^{(i)} (z^{(i)})-\prod_{i=1}^{r+s} \phi_n^{(i)} (J_iz^{(i)})\right|^p\,
dz^{(1)}\cdots dz^{(r+s)} \\
   & = & \lim_{n \to \infty } \int_{\mathbb{R}^{d_2} } \cdots \int_{\mathbb{R}^{d_{r+s}} } \left|\prod_{i=2}^{r+s}  \phi_n^{(i)} (z^{(i)})-\prod_{i=2}^{r+s} \phi_n^{(i)} (J_iz^{(i)})\right|^p\,
dz^{(2)}\cdots dz^{(r+s)}.
  \end{eqnarray*}
   Iterating the procedure for all Jordan block that corresponds to eigenvalues on the unit circle, we reduce the proof
  to the cases of Jordan blocks $J_i$ that are expansive or  $J_i^{-1}$ are expansive.
  Indeed, without loss in the generality we can assume that we have only two Jordan blocks, one  is expansive and the other has its inverse expansive. The proof finishes using Lemma~\ref{lemma:expansivecontractive}.
  \end{proof}

  Once we have constructed a minimizing sequence in Lemma \ref{por.arriba}, the proof of Theorem \ref{teo.cota.por.abajo} follows immediately.
\begin{proof}[Proof of Theorem \ref{teo.cota.por.abajo}]
By Lemma \ref{lemma.geq}, Lemma \ref{cota.por.arriba} and Lemma \ref{por.arriba} the statement follows.
\end{proof}

\section{Decay estimates for the evolution problem}
\label{Sect.estim.evol}
\setcounter{equation}{0}

Let us consider a nonnegative solution $u(x,t)$ to
\eqref{eq.parabolica}. Note that, since the kernel
is nonnegative, there is a comparison principle for this problem. Therefore, $- u^-(x,t) \leq u(x,t) \leq u^+ (x,t)$, where $u^-$ and $u^+$ are the solutions with initial condition the negative and the positive part of $u_0$ respectively. Hence, for the proof of Theorem \ref{teo.decaimiento.intro}, we may assume that the solutions are nonnegative.

\begin{proof}[Proof of Theorem~\ref{teo.decaimiento.intro}]
First, we assume that $r>1$ and $p>2$. Multiplying equation \eqref{eq.parabolica} by $u^{r-1} (x,t)$ and integrating on the $x$ variable, we obtain the following
\[
\begin{array}{l}
\displaystyle \frac1r \frac{d}{dt} \int_{\rr^d} u^{r} (x,t) \, dx  \displaystyle =
\int_{\rr^d} \int_{\rr^d} K(x,y) |u(y,t)- u(x,t)|^{p-2} (u(y,t)- u(x,t))  u^{r-1}(x,t)\, dy\, dx
\\[8pt]
 \displaystyle = - \frac12
\int_{\rr^d} \int_{\rr^d} K(x,y) |u(y,t)- u(x,t)|^{p-2} (u(y,t)- u(x,t))  (u^{r-1}(y,t) - u^{r-1}(x,t) )
\, dy\, dx.
\end{array}
\]
Now, we use that, for $p>2$ and $r>1$ there is a constant $\widetilde{C}(p,r)$ such that
$$
|u(y,t)- u(x,t)|^{p-2}\, (u(y,t)- u(x,t))\, (u^{r-1}(y,t) -
u^{r-1}(x,t)) \geq \widetilde{C}(p,r) \, |u^\alpha (y,t)- u^\alpha (x,t)|^{p},
$$
with
$$
\alpha  =\frac{r+p-2}{p}.
$$
Indeed, one can check that is a constant $\widetilde{C}(p,r)$ such that
$$
|1- z|^{p-2}\, (1- z)\, (1 -
z^{r-1}) \geq \widetilde{C}(p,r) \, |1- z^\alpha |^{p},
$$
for all $0 \leq z \leq  1$, because
$$
\lim_{z \to 1}\frac{|1- z|^{p-2}\, (1- z)\, (1 -
z^{r-1})}{|1- z^\alpha |^{p}} > 0.
$$

We obtain that, for some constant $C=C(p,r)$, the following holds:
\begin{align*}
\frac{d}{dt} \int_{\rr^d} u^{r} (x,t) \, dx & \leq - C \int_{\rr^d} \int_{\rr^d} K(x,y)\, |u^\alpha (y,t)- u^\alpha (x,t)|^{p} \, dy\, dx \\
& \leq - C \lambda_1\, (\rr^d) \int_{\rr^d} |u (x,t)|^{\alpha p}
\,  dx .
\end{align*}

Now we observe that for $p>2$, we have
$
\alpha p = r + p - 2 > r
$
and therefore we can use the interpolation inequality
$$
\| u (\cdot, t) \|_{L^r (\rr^d)} \leq \| u (\cdot, t) \|^\theta_{L^1 (\rr^d)}
\| u (\cdot, t) \|^{1-\theta}_{L^{\alpha p } (\rr^d)}
$$
with
$$
\frac{1}{r} = \theta + \frac{1-\theta}{\alpha p}
$$
that is,
$$
\theta = \frac{1}{r} \frac{\alpha p -r}{\alpha p -1}, \qquad 1-\theta = \frac{\alpha p}{r} \frac{r-1}{\alpha p -1}.
$$
Hence, using that the $L^1(\rr^d)$-norm of $u( \cdot, t)$ does not increase
$$
\| u (\cdot, t) \|_{L^1 (\rr^d)} \leq \| u_0 (\cdot) \|_{L^1 (\rr^d)}
$$
we obtain that there exists $k_1>0$ such that
$$
\frac{d}{dt} \| u (\cdot, t) \|^r_{L^r (\rr^d)} \leq - k_1
 \left( \| u (\cdot, t) \|^{r}_{L^r (\rr^d)} \right)^{\frac{\alpha
 p}{(1-\theta) r}}.
$$
Then,
$$
\| u (\cdot, t) \|^r_{L^r (\rr^d)} \leq k_2 \, t^{- \frac{(1-\theta)
r}{\alpha p - (1-\theta) r}},
$$
that is,
$$
\| u (\cdot, t) \|_{L^r (\rr^d)} \leq k_3 \, t^{-
\frac{r-1}{p -2}},
$$
for some $k_3>0$ depending on $u_0$, $p$, $r$ and $K$.

Now, for $r>1$ and $1\leq p \leq 2$, using the comparison principle for \eqref{eq.parabolica} (see
\cite[Theorem 6.37]{libro}), we have
$$
\| u (\cdot, t) \|_{L^\infty (\rr^d)} \leq \| u_0 (\cdot) \|_{L^\infty
(\rr^d)}.
$$
Hence, for any $q\geq p$ there exists a constant $\widetilde{C}=2\, \|u_0\|_{L^\infty(\rr^d)}^{{\alpha}(q-p)}$ such that
$$
\widetilde{C}\, |u^\alpha (y,t)- u^\alpha (x,t)|^{p} \geq |u^\alpha (y,t)-
u^\alpha (x,t)|^{q}.
$$
Therefore, let us choose $q$ such that
$
\alpha q = r
$
and perform the same computation as in the previous case. We obtain
$$
\begin{array}{rl}
\displaystyle \frac{d}{dt} \int_{\rr^d} u^{r} (x,t) \, dx &  \displaystyle
\leq - C
\int_{\rr^d} \int_{\rr^d} K(x,y) |u^\alpha (y,t)- u^\alpha (x,t)|^{q}
\, dy\, dx \\[8pt]
& \displaystyle \leq - \gamma
\int_{\rr^d} |u (x,t)|^{r}
\,  dx .
\end{array}
$$
Hence an exponential decay of $u$ in $L^r(\rr^d)$-norm follows
$$
\displaystyle \int_{\rr^d} u^r (x,t) \, dx \leq
\left(\int_{\rr^d}u^r(x,0) \, dx \right) \cdot e^{-\gamma t},
$$
for some $\gamma>0$ depending on $K$, $p$, $r$ and $u_0$.

Finally, if $r=1$ we just have to multiply by sgn$(u)$ and the proof follows similarly as in the previous cases.
\end{proof}

\section{The case $p=\infty$}
\label{Sect.infty}
\setcounter{equation}{0}

Let us consider for any compact support function $u$ the quantities:
$$Q_p(u)=\frac{\displaystyle
\int_{\rr^d}\int_{\rrd} K(x,y)|u(x)-u(y)|^pdxdy}{\displaystyle\int
_{\rrd}|u|^p(x)dx}$$
and
$$Q_\infty(u)=\frac{ \| u(x) - u(y)\|_{L^\infty (x,y \in supp (u); \, K(x,y) >0 )} }{\| u
    \|_{L^\infty (supp (u))} } .$$
We set
\begin{equation}\label{eq.infty}
   \lambda_{1,\infty}(\rr^d)= \inf \{Q_\infty (u): u \in L^\infty(\rr^d), \mbox{ compactly supported} \}  .
\end{equation}
We remark that  taking $u\equiv 1$ we obtain $Q_\infty(u)=0$
and therefore in the definition of $ \lambda_{1,\infty} (\rr^d)$ we
have to consider functions $u$ that are compactly supported in
$\rr^d$.
\begin{lemma} \label{lemma.infty.1} The first eigenvalue for
$p=\infty$, \eqref{eq.infty}, is a bound for the limit of the
first eigenvalues as $p\to \infty$,
$$
\limsup_{p \to \infty } [\lambda_{1,p} (\rr^d)]^{1/p} \leq \lambda_{1,\infty}
(\rr^d).
$$
\end{lemma}
\begin{proof} Fix $u\in L^\infty (\rr^d)$ compactly supported such
that
$$
Q_\infty (u) - \lambda_{1,\infty} (\rr^d) < \varepsilon.
$$
Now we observe that
$$
\lim_{p\to \infty} [Q_p (u)]^{1/p} =
\lim_{p\to \infty} \left( \frac{\displaystyle
\int_{\rr^d}\int_{\rrd} K(x,y)|u(x)-u(y)|^pdxdy}{\displaystyle\int
_{\rrd}|u|^p(x)dx}\right)^{1/p} = Q_\infty (u) <
\lambda_{1,\infty} (\rr^d) +\varepsilon.
$$
Hence
$$
\limsup_{p \to \infty} [\lambda_{1,p} (\rr^d)]^{1/p} \leq \lambda_{1,\infty}
(\rr^d).
$$
\end{proof}
\begin{lemma} \label{lema.2}
Let $K$ be such that its support satisfies
$$
{\rm supp} (K) \subset \{ |x-a(y)|\leq 1\} \cup \{ |y-a(x)|\leq 1\}.
$$
Then
$
\lambda_{1,\infty}
(\rr^d) = 0.
$
\end{lemma}
\begin{remark}
For kernels defined by \eqref{forma.nucleo}, we have assumed that the function $\psi$ is supported in the unit ball, so the hypothesis of the lemma holds. If the support of $\psi$ is any compact set in the ball with radius $R$, the result remains true replacing $1$ by $R$.
\end{remark}
\begin{proof} Given $\epsilon >0$, we just have to construct a function $u_\epsilon$,
compactly supported such that
$$
 \frac{ \| u_\epsilon(x) - u_\epsilon(y)\|_{L^\infty (x,y \in supp (u_\epsilon); \, K(x,y) >0 )} }
 {\| u_\epsilon
    \|_{L^\infty (supp (u_\epsilon))} } \leq \epsilon.
$$
To this end let us start with
$$
u_\epsilon (x) =1, \qquad \mbox{in } |x|\leq 1,
$$
Next, we let
$$
u_\epsilon (x) = 1- \epsilon, \qquad \mbox{in } 1<|x|\leq R_a^1
$$
where $R_a^1$ is such that $|x-a(y)| >1$ if $|x|\leq 1$ and $|y|
\geq R_a^1$. In this way if we let $|x|\leq 1$ the points $|y|
\geq R_a^1$ are such that $(x,y) \not\in supp(K)$, hence for
$|x|\leq 1$,
$$
\| u_\epsilon(x) - u_\epsilon(y)\|_{L^\infty (y ; \, (x,y) \in \textrm{supp} (u_\epsilon) \textrm{ and } \, K(x,y) >0
)} \leq \epsilon.
$$
Now, we continue with
$$
u_\epsilon (x) = 1-2\epsilon \qquad \mbox{in } R_a^1 < |x| \leq
R_a^2
$$
where $R_a^2$ is such that $|x-a(y)| >1$ if $|x|\leq R_a^1$ and
$|y|
\geq R_a^2$. Analogously as before we get
$$
\| u_\epsilon(x) - u_\epsilon(y)\|_{L^\infty (y ; \, (x,y) \in \textrm{supp} (u_\epsilon) \textrm{ and } \, K(x,y) >0
)} \leq \epsilon,
$$
for $|x| \leq R_a^1$.

Iterating this procedure a finite number of times, $N=
[1/\epsilon]+1$ we get a compactly supported function, whose
support is included in $|x|\leq R_a^N$. Moreover, the $L^\infty-$norm is one and
$$
\| u_\epsilon(x) - u_\epsilon(y)\|_{L^\infty (x,y \in supp (u_\epsilon); \, K(x,y) >0
)} \leq \epsilon.
$$
The proof is now complete.
\end{proof}
As an immediate consequence of the previous results we get Theorem \ref{lim.p.0.intro}.
\begin{proof}[Proof of  Theorem \ref{lim.p.0.intro}]
We just observe that
$$
0\leq \liminf_{p \to \infty } [\lambda_{1,p} (\rr^d)]^{1/p}\leq
\limsup_{p \to \infty } [\lambda_{1,p} (\rr^d)]^{1/p}
\leq
\lambda_{1,\infty} (\rr^d)=0.
$$
and the proof is finished
\end{proof}

\medskip

 {\bf Acknowledgements.}

L. I. Ignat is partially supported by grants PN II-RU-TE 4/2010  and PCCE-55/2008
of the Romanian National Authority for
Scientific Research, CNCS--UEFISCDI, MTM2011-29306-C02-00, MICINN,
Spain and ERC Advanced Grant FP7-246775 NUMERIWAVES.

D. Pinasco is partially supported by grants ANPCyT PICT 2011-0738 and CONICET PIP 0624.

J. D. Rossi and A. San Antolin are partially supported by
DGICYT grant PB94-0153 MICINN, Spain.

\bigskip


\end{document}